\documentclass[11pt]{amsart}

\usepackage{amsmath,amsthm}
\usepackage{amsfonts}
\usepackage{amssymb}
\usepackage{xcolor}
\usepackage{url}
\usepackage{graphicx}
\newcommand{\thtw}{\frac{\theta}{2}}

\newcommand{\delete}[1]{}

\theoremstyle{plain}
\newtheorem{theo}{Theorem}[section]

\theoremstyle{remark}
\newtheorem{rem}[theo]{Remark}

\allowdisplaybreaks

\begin{document}

\title[Observable clade sizes]{The minimal observable clade size of exchangeable coalescents}

\date{\today}
 
\author{Fabian Freund}
  \address{Crop Plant Biodiversity and Breeding Informatics Group (350b), Institute of Plant Breeding, Seed Science and Population Genetics,   University of Hohenheim, Fruwirthstrasse 21, 70599 Stuttgart, Germany}
 \email{fabian.freund@uni-hohenheim.de}
 
  \author{Arno Siri-J\'egousse}
  \address{Departamento de Probabilidad y Estad\'istica,
IIMAS, Universidad Nacional Aut\'onoma de M\'exico, Mexico City, Mexico. }
 \email{arno@sigma.iimas.unam.mx}

\begin{abstract}
For $\Lambda$-$n$-coalescents with mutation, we analyse the size $O_n$ of the partition block of $i\in\{1,\ldots,n\}$ at the time where the first mutation appears on the tree that affects $i$ and is shared with any other $j\in\{1,\ldots,n\}$. We provide asymptotics of $O_n$ for $n\to\infty$ and a recursion for all moments of $O_n$ for finite $n$.
This variable gives an upper bound for the minimal clade size \cite{Blum2005}, which is not observable in real data.
In applications to genetics, it has been shown to be useful to lower classification errors in genealogical model selection \cite{FSStat}. 
\end{abstract}

\maketitle

\section{Introduction}
The potential for adaption of organisms to diverse environments is based on their genetic diversity. Moreover, the specific historic pattern of adaptation and demography leaves distinct marks in the genetic diversity of a sample taken from a population of said organisms. When observing the genetic diversity of a single non-recombining part of the genome, the diversity can be described by the inheritance pattern of the mutations on the genealogical tree of the sample, usually given by a Poisson point process on the genealogy. Modelling the genealogy is thus an important aspect of modelling genetic diversity. Usually, the exact genealogy is not known and cannot be reconstructed perfectly from observed genetic data (an example is provided later). Thus, genealogy models are usually defined as random variables on the set of possible genealogical trees.\\  
Here, we are concerned with the genealogical tree of $n$ alleles, i.e. the genetic information of a sample of size $n$ from a genomic region.  Coalescent theory provides a rich class of genealogical tree models for a sample of $n$ alleles as much as elegant tools and a convenient setting up for statistical inference. In particular, Kingman's coalescent \cite{Kingman1982b} and the larger family of coalescents with multiple collisions \cite{Pitman1999, Sag99} were widely studied in the past decades.
This class of Markov processes on the set of partitions of $[n]:=\{1,\ldots,n\}$ is characterized by a finite measure $\Lambda$ on $[0,1]$, justifying their name of $\Lambda$-$n$-coalescents.
If a $\Lambda$-$n$-coalescent has $b$ blocks, any given $k$-tuple of them will merge at rate
$$\lambda_{b,k}=\int_0^1x^{k-2}(1-x)^{b-k}\Lambda(dx)$$
and the rate for the next coalescence event is
$$\lambda_b=\sum_{k=2}^b\binom{b}{k}\lambda_{b,k}.$$
The starting partition is $\{1\},\ldots,\{n\}$. The genealogical tree is recovered from the partition-valued process by first starting $n$ branches from the leaves $1,\ldots,n$. Then, any merger of partition blocks corresponds to a joining of branches in a node (ancestor), where a single new branch starts. Partition blocks correspond to branches in the genealogical tree, the time a partition block is not merged  gives the length of this branch.  We refer to \cite{GIM14} for a survey.\\
Genealogical trees of the alleles in a genetic region come with an interpretation of relatedness and genetic similarity: An allele $i$ is more closely related to allele $j$ than to allele $l$ if the common ancestor of $(i,j)$ appears more recently than the common ancestor of $(i,l)$, while the path lengths between leaves (alleles) measure the time available to accumulate mutations that decrease genetic similarity. Several statistics aim to capture these aspects and their biological meaning. For instance, the {\it minimal clade size} of an allele $i$ gives the number of closest relatives of $i$, see \cite{Blum2005}. Another example is the {\it length of the external branch} of $i$, i.e. the waiting time for the first merger of $i$, which gives a measure of the genetic uniqueness of an individual \cite{RB04}. 
The mathematical properties of the minimal clade size \cite{Blum2005, Freund2014, SY16, Freund2017size}, the length of an external branch \cite{Blum2005, CNR07, DFSY}, as much as the family (partition block) sizes at this time \cite{Blum2005, SY18}, have been analysed recently.
In these works, asymptotic and exact behaviours are obtained for various examples of exchangeable coalescents.\\
However, these statistics cannot be observed directly from the genetic data. We will illustrate this for the minimal clade. By a clade we denote the set of all alleles that share a specific ancestor, and the minimal clade of $i$ is the smallest clade including $i$. Assume the infinite-sites model of mutation, each mutation causes a change at a different position in the genomic region. Further assume that we know the ancestral state at the genomic region, i.e. we can identify mutations as changes compared to the ancestral state. Any clade can only be observed if there is at least one mutation that all its members share. This mutation is inherited from the common ancestor, thus has to be placed on the branch that connects this ancestor further towards the root of the genealogy (the most recent common ancestor of the whole sample). Thus, we can only observe a clade if there is a mutation on the branch directly above of the ancestor defining this clade.\\
Instead of looking at the minimal clade of an allele $i$, one could consider the smallest clade which includes $i$ that can be observed from the data. We considered the sizes of these clades for all alleles sampled, the {\it minimal observable clade sizes}, in \cite{FSStat}. There we could show that they provide an additional set of statistics that faciliates the inference of a well-fitting genealogy model when coupled with standard statistics of genetic diversity as the site frequency spectrum.\\
In this article, we study some mathematical properties of the minimal observable clade size for an individual $i$. Its asymptotic behaviour for any $\Lambda$-$n$-coalescent for $n\to\infty$ as well as a recursion for all moments for finite $n$ are established. For the Bolthausen-Sznitman coalescent, which provides a somewhat universal genealogical model for populations under strong selection, see e.g. \cite{Neher2013a}, \cite{Desai2013}, \cite{Schweinsberg2017}, we can show that the minimal observable clade size is asymptotically Beta-distributed.

\section{A formal definition of the minimal observable clade size}
Let $[n]:=\{1,\ldots,n\}$ for $n\in\mathbb{N}$, $[n]_0=[n]\cup\{0\}$. For any $\Lambda$-$n$-coalescent and a sampled allele $i\in[n]$, define 
\begin{itemize}
\item $\mathcal{C}_{n,i}(t)$ as the partition block $i$ is in at time $t$ (a size-biased pick of a block of the $n$-coalescent at time $t$)
\item $\kappa_n$ as the total number of jumps of the $\Lambda$-$n$-coalescent, $\kappa_{n,i}$ as the total number of jumps (the block of) $i$ participates in
\item $K_{n,i}(0)(=0),K_{n,i}(1),\ldots,K_{n,i}(\kappa_{n,i})(=\kappa_n)\in[\kappa_{n}]_0$ as the successive indices of jumps in the $\Lambda$-$n$-coalescent in which the block of $i$ is involved   
\item $\mathcal{C}_{n,i}[k]$ as the partition block $i$ is in at the time of its $k$th jump $K_{n,i}(k)$, $k\in[\kappa_{n,i}]_0$. $\mathcal{C}_{n,i}[0]=\{i\}$, $\mathcal{C}_{n,i}[1]$ is the minimal clade of $i$, $\mathcal{C}_{n,i}[\kappa_{n,i}]=[n]$. 
\end{itemize}
Given the $\Lambda$-$n$-coalescent tree, we set mutations on its branches via a homogeneous Poisson point process with rate $\thtw$. 
Mutations are interpreted under the infinite sites model, each mutation hits a site not hit by any other mutation, producing a new type. The new type is called derived type in contrast to the ancestral type of the most recent ancestor of the sample. Mutations on external branches are affecting only one individual, we will call these private mutations; they can also be referred to as singleton mutations. All other mutations are called non-private mutations. Since we are interested in the mutations carried by individual $i$, we have to record the mutations from $t=0$ to the time back to the most recent common ancestor of the sample (the root of the genealogy) on the path of $i$.  Let $T^{(i)}_n$ be the waiting time until the first (youngest) mutation on the path of $i$ that is non-private, i.e. does not fall on the external branch which ends in $i$ (which has length $E^{(i)}_n$). If we continue the path of $i$ after reaching the most recent common ancestor as a single ancestral line indefinitely (which we will do from now on), we have 
\begin{equation}\label{eq:Tnp}
T^{(i)}_n\stackrel{d}{=}E^{(i)}_n+M^{(i)}
\end{equation}
for an independent exponential random variable $M^{(i)}$ with rate $\thtw$. Let 
\begin{equation}\label{def:Ln}
L^{(i)}_n:=\max\{k,\mbox{ jump $K_{n,i}(k)$ happens earlier than $T^{(i)}_n$}\}\in[\kappa_{n,i}],
\end{equation}
i.e. the $L^{(i)}_n$th jump that $i$ participates in is the last jump of it before $T^{(i)}_n$. The \textit{minimal observable clade of $i$} is then given by 
$$
\mathcal{C}_{n,i}[L^{(i)}_n] =\{j\in[n]:\mbox{$j$ shares all non-private mutations of $i$}\}
$$
The definitions are equivalent since all non-private mutations of $i$ are inherited from the youngest ancestor of $i$ that bears at least one mutation on the branch connecting it to the next older ancestor. If $i$ has no non-private mutations, $\mathcal{C}_{n,i}[L^{(i)}_n]=[n]$ almost surely, since in this case $T^{(i)}_n$ is larger than the time back to the most recent common ancestor.\\
The statistic we are interested in is the size of the minimal observable clade of an allele $i$ 
$$
O_n(i):=|\mathcal{C}_{n,i}[L^{(i)}_n]|.
$$
See Figure \ref{fig:onmn} for an example. Due to exchangeability, the distribution of $O_n(i)$ does not depend on $i$, we can even choose $i$ randomly without changing the distribution. For ease of notation, we fix the allele we are interested in to allele 1 and abbreviate $O_n:=O_n(1)$.
\begin{rem}
Since the partition block including $i$ can only increase in size over time, the minimal clade $\mathcal{C}_{n,i}[1]$ is a subset of the minimal observable clade $\mathcal{C}_{n,i}[L^{(i)}_n]$ for $i\in[n]$. Thus, the size $M_n(i)$ of the minimal clade of $i\in[n]$ satisfies $M_n(i)\leq O_n(i)$. See Figure \ref{fig:onmn} for an example.
\end{rem}
\begin{figure}[ht]
\includegraphics[scale=1]{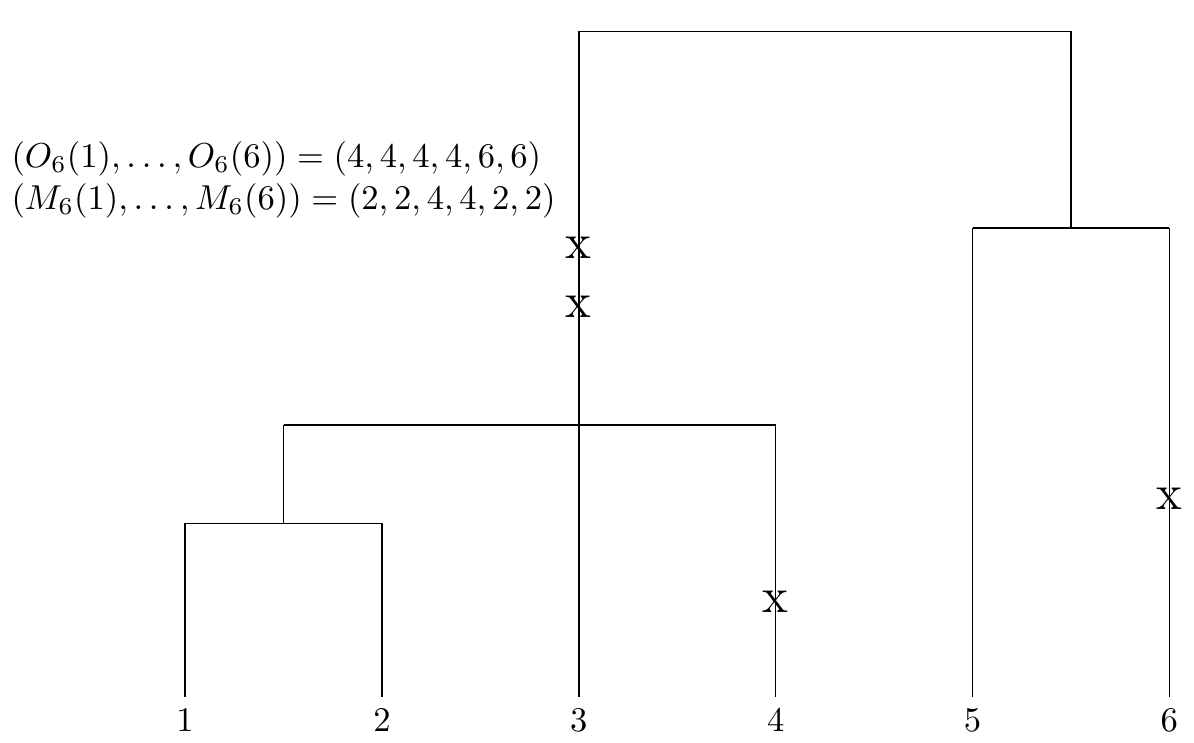}
\caption{Genealogical tree and its minimal observable clade sizes $O_6(i)$ and minimal clade sizes $M_6(i)$ for $i\in[6]$. x denotes a mutation.}
\label{fig:onmn}
\end{figure}

\section{Asymptotics of the observable clade size}
Asymptotically for sample size $n\to\infty$, the probabilistic structure of $O_n$ simplifies considerably. 
First, we focus on coalescents without dust, a class which includes Beta($2-\alpha$,$\alpha$)-$n$-coalescents for $\alpha\in(1,2)$ \cite{Sch03}, Kingman's $n$-coalescent ($\Lambda=\delta_0$) and the Bolthausen-Sznitman $n$-coalescent ($\Lambda$ uniform on $[0,1]$). A $\Lambda$-$n$-coalescent has dust if and only if $\mu_{-1}:=\int_{[0,1]} x^{-1} \Lambda(dx)<\infty$, see \cite{Pitman1999}.\\ 

\begin{theo}\label{prop:convtoS}
Let $O_n$ be defined for $\Lambda$-$n$-coalescents such that $\mu_{-1}=\infty$ (without dust) and with mutation rate $\thtw$. We have \begin{equation}\label{limOn}\frac{O_n}{n}\stackrel{a.s.}{\to} S\end{equation} for $n\to\infty$ with $S>0$ a.s.. $S$ is distributed as the size of the block of individual 1 (alternatively a size-biased pick of a block size) at a random time $M\stackrel{d}{=}Exp\left(\thtw\right)$. The distribution of $S$ is uniquely determined by its moments
\begin{equation}\label{eq:Smoments_nodust}
E(S^k)=1-\sum_{r=2}^{k+1} a_{k+1,r} \frac{\thtw}{\lambda_r+\thtw}, 
\end{equation}  
where $\lambda_r$ is the total rate of the $\Lambda$-coalescent in a state with $r$ blocks and $a_{k+1,r}$ is a rational function of $\lambda_2,\ldots,\lambda_{k+1}$, defined as in \cite[Prop. 29]{Pitman1999}. In particular,
$$E(S)=\frac{\Lambda([0,1])}{\Lambda([0,1])+\thtw},\ E(S^2)=1-\frac{3}{2}\frac{\thtw}{\Lambda([0,1])+\thtw}+\frac{1}{2}\frac{\thtw}{\lambda_3+\thtw}.$$ 
\end{theo}  
 
\begin{proof}
Let $E^{(1)}_n$ be the waiting time for the first collision of individual 1. By the consistency of the $n$-coalescents, we have
 $E^{(1)}_2\geq E^{(1)}_3\geq \ldots$. By a slight adaptation of \cite[Prop. 26]{Pitman1999}, we see that $E^{(1)}_n\stackrel{d}{\to}0$ for $n\to\infty$. Since $(E^{(1)}_n)_{n\in\mathbb{N}}$ is monotonically decreasing, this convergence also holds almost surely.\\

All mutations of individual 1 on any $n$-coalescent lie on the path of leaf 1 to the root of the coalescent tree and are consistent for different values of $n$ (any mutation on the path to the root in the $m$-coalescent is also a mutation on this path in every $n$-coalescent with $n>m$), since for $m<n$, the $m$-coalescent (seen as a tree) is the subtree of the $n$-coalescent which is spanned by the leaves $[m]$, including mutations. Thus, we can represent the mutations of individual 1 on the $n$-coalescent by one common homogeneous Poisson process for all $n$, independent of the $n-$coalescents, on $[0,\infty)$ with rate $\thtw$ %(ignoring points which are bigger than the time back to the MRCA for each $n$-coalescent). 
This gives a new representation for $T_n^{(1)}$, it is the smallest Poisson point $T$ with $T\geq E^{(1)}_n$. Let $M$ be the smallest Poisson point overall. Since $E^{(1)}_n\to 0$ a.s. for $n\to\infty$, for any realisation of the coalescent there exists a $n_0\in\mathbb{N}$ s.t. $E^{(1)}_m<M$ a.s. for all $n\geq n_0$. This shows that $T_n^{(1)}=M$ a.s. for $m\geq n_0$, which implies 
\begin{equation}\label{eq:Snodust}
\lim_{n\to\infty}\frac{O_n}{n}=\lim_{n\to\infty}\frac{\mathcal{C}_{n,1}(M)}{n}=f_1(M)(=f_1(M^{(1)})) \mbox{ a.s.},
\end{equation}
where $f_1(t)$ is the (asymptotic) frequency of the block individual 1 is in at time $t\geq 0$ in the $\Lambda$-coalescent (with values in $\mathbb{N}$, see \cite{Pitman1999}). The existence of the limit follows from Kingman's correspondence, since the coalescent stopped at the random time $M$ (independent of the coalescent) gives an exchangeable partition of $\mathbb{N}$. Since we have a coalescent without dust, we have no singleton blocks a.s. at any time $t>0$ and a (potentially infinite) number of blocks with a.s. positive frequencies, again due to Kingman's correspondence. This shows $f_1(M)>0$ a.s..
 Due to exchangeability, the distribution of $f_1(M)$ is the same as if we would make a size-biased pick from all blocks present.\\
{Denoting $f_1(M)$ by $S$}, consider the moments 
$$
E(S^k)=\int_0^{\infty} E((f_1(t))^k) \thtw e^{-\thtw t} dt.
$$
From \cite[Eq. (50) and Prop. 29]{Pitman1999} we see that $E((f_1(t))^k)=1-\sum_{r=2}^{k+1} a_{k+1,r} e^{-\lambda_r t}$ due to a connection to the exchangeable partition function of the coalescent at time $t$. Thus, we have
\begin{equation}\label{eq:moments_nodust}
E(S^k)=1- \sum_{r=2}^{k+1} a_{k+1,r}\int_0^{\infty} e^{-\lambda_r t}\thtw e^{-\thtw t} dt=1-\sum_{r=2}^{k+1} a_{k+1,r}\frac{\thtw}{\lambda_r+\thtw},
\end{equation}
which is Eq. \eqref{eq:Smoments_nodust}. Using the explicit values $a_{2,2}=1,a_{3,2}=\frac{3}{2},a_{3,3}=-\frac{1}{2}$ (essentially from \cite[Eq. (39),(40)]{Pitman1999}) and $\lambda_2=\Lambda([0,1])$ yields the first two moments. Since $S$ takes values in $[0,1]$, its distribution is uniquely determined by its moments.    
\end{proof}
%\begin{rem}
%If an analogon of \cite[Eq. (50) and Prop. 29]{Pitman1999} exists also for $\Xi$-$n$-coalescents, our proof would work also in this case \fabi{This really does sound shifty and not serious, let's omit it...}
%\end{rem} 
For the special case of the Bolthausen-Sznitman coalescent, the law of $S$ can be identified
\begin{theo}\label{SBSC}
For the Bolthausen-Sznitman $n$-coalescent, 
$$\frac{O_n}{n}\stackrel{a.s.}{\to} S,$$
for $n\to\infty$ where $S\stackrel{d}{=}Beta\left(\frac{1}{1+\thtw},\frac{\thtw}{1+\thtw}\right)$.
%where $S\stackrel{d}{=}\frac{X}{X+Y}$, where $X$ and $Y$ are independent and Gamma-distributed with parameters $\left(\frac{1}{1+\thtw},1\right)$ and $\left(\frac{\thtw}{1+\thtw},1\right)$.
\end{theo} 
\begin{proof}
\cite[Corrolary 16]{Pitman1999} shows that for the Bolthausen-Sznitman coalescent, $(f_1(t))_{t\geq 0}$ jumps at independent standard exponential times with ranked jump sizes given by a Poisson-Dirichlet distribution with parameters $(0,1)$. The set of jump times is independent of the set of jump sizes. Comparing this with Eq. \eqref{eq:Snodust}, we see that to compute $S\stackrel{a.s.}{=}f_1(M^{({1})})$, we need to sum the sizes of all jumps of $f_1$ that happen before or at $M^{({1})}$. Consider the jump times $(T_k)_{k\in\mathbb{N}}$ of $f_1$ ordered according to the rank of their jump sizes $(J_k)_{k\in\mathbb{N}}$. Define $B_k:=1_{\{T_k\leq M^{({1})}\}}$. Hence, we have
$P(B_k=1)=(1+\thtw)^{-1}$. We can now express
\begin{equation}\label{eq:f1_BSZ_PD}
S\stackrel{a.s.}{=}\sum_{k\in\mathbb{N}} B_kJ_k.
\end{equation}
In other words, $S$ can be seen as summing up a random thinning of a standard Poisson-Dirichlet distributed random variable.\\
The random variable $S$ is Beta distributed. To see this, we will use the construction of the $PD(0,\theta')$ distribution from \cite{kingman1975random}, which is also summarised in \cite[Section 4.11]{arratia2003}. Consider the points $\mathcal{P}:=(P_k)_{k\in\mathbb{N}}$ of a Poisson point process on $[0,\infty)$ with mean measure $\nu_{\theta'}:=\frac{\theta' e^{-x}}{x}dx$. Then, the size-ordered and normalized points  $\left(\frac{P_{[k]}}{P}\right)_{k\in\mathbb{N}}$ with $P=\sum_{i\in\mathbb{N}}P_k$ have the Poisson-Dirichlet distribution and are independent of 
\begin{equation}\label{eq:poissum}
P\stackrel{d}{=}Gamma(\theta',1),
\end{equation}
where $Gamma(\alpha,\beta)$ is the Gamma distribution with shape parameter $\alpha$ and rate $\beta$.\\
We choose $\theta'=1$ and make the correspondence between the ranked and normalised points $(P_{k})_{k\in\mathbb{N}}$ and the jump sizes $(J_k)_{k\in\mathbb{N}}$. To express Eq. \eqref{eq:f1_BSZ_PD}, we give each point $P_k$ a mark $B_k\in\{0,1\}$. Marks are independent from $(P_k)_{k\in\mathbb{N}}$ and from one another. We set the probability to be marked to $m:=P(B_k=1)=(1+\thtw)^{-1}$ for all $k\in\mathbb{N}$. $(P_k,B_k)_{k\in\mathbb{N}}$ is a marked Poisson process. The Colouring Theorem \cite[Section 5.1]{Kingman1993} shows that all points $P_k$ with marks 1 form a Poisson point process $\mathcal{P}_1$ with mean measure $m\nu_{1}$, while all points with mark 0 form a Poisson point process $\mathcal{P}_0$ with $(1-m)\nu_{1}$.\\
We can now alternatively express \eqref{eq:f1_BSZ_PD} as
$$S\stackrel{d}{=}\frac{\sum_{p\in\mathcal{P}_1}p}{\sum_{p\in\mathcal{P}_1}p+\sum_{p\in\mathcal{P}_0}p}=:\frac{X}{X+Y}.$$         
where $X$ and $Y$ are independent due to the independence of $\mathcal{P}_0$ and $\mathcal{P}_1$. Since the mean measures of $\mathcal{P}_0$ and $\mathcal{P}_1$ are of the form $\frac{\theta' e^{-x}}{x}dx$ with $\theta'$ equal to $m$ and $1-m$, Eq. \eqref{eq:poissum} yields $X\stackrel{d}{=}Gamma(m,1)$ and $Y\stackrel{d}{=}Gamma(1-m,1)$. Thus, $\frac{X}{X+Y}$ is Beta-distributed with parameters $m$ and $1-m$.
\end{proof}

\begin{rem}
Theorem \ref{prop:convtoS} can be generalised for some time-changed $\Lambda$-$n$-coalescents without dust, which appear when modelling genealogies in Cannings models with moderate fluctuations in population size, see \cite{griffiths1994sampling}, \cite{Matuszewski2017}, \cite{Spence1549} and \cite{freund2019cannings}. 
Let $(\Pi^{(n)}_{g(t)})_{t\geq 0}$ be a time-changed $\Lambda$-$n$-coalescent, where $g(t):=\int_0^{t}\mu(s)ds$  with continuous $\mu:[0,+\infty)\mapsto [0,+\infty)$, which includes some time changes proposed for $\Lambda$-$n$-coalescents in the references above. Observe that $g$ is continuous, monotone and invertible with differentiable inverse. The time-changed $\Lambda$-$n$-coalescent is still exchangeable. The almost sure convergence of $n^{-1}O_n$ for the time-changed $\Lambda$-$n$-coalescent works analogously as in Theorem \ref{prop:convtoS}. The time of the first merger of individual 1 is $g(E^{(1)}_n)$, thus also converges to 0 almost surely. The waiting time $M'_1$ for the first mutation on the path of 1 to the root is an $Exp(\thtw)$-distributed random variable, but on the time-changed path of $(\Pi^{(n)}_{g(t)})_{t\geq 0}$. Thus, the limit of $n^{-1}O_n$ for the time-changed $\Lambda$-$n$-coalescent is the frequency of the block containing 1 at time $M'_1$ in $(\Pi^{(n)}_{g(t)})_{t\geq 0}$. This can also be expressed as $f_1(g(M'_1))$, where $f_1$ is said frequency in the $\Lambda$-$n$-coalescent $(\Pi^{(n)}_{t})_{t\geq 0}$. The distribution of $g(M'_1)$ is given by
$$P(g(M'_1)\leq t)=P(M'_1\leq g^{-1}(t))=1-e^{-\thtw g^{-1}(t)},$$
which has density $t\mapsto \thtw \mu(g^{-1}(t))^{-1} e^{-\thtw g^{-1}(t)}$. Analogously to \eqref{eq:moments_nodust} we can thus express, in terms of the $a_{k,r}$ from Theorem \ref{prop:convtoS}, the $k$th moment of $n^{-1}O_n$ for the $n$-coalescent with exponential growth as  
\begin{align*}
E(S^k)=& 1- \sum_{r=2}^{k+1} a_{k+1,r}\int_0^{\infty} e^{-\lambda_r t} \thtw \mu(g^{-1}(t))^{-1} e^{-\thtw g^{-1}(t)}dt\\
=& 1- \thtw\sum_{r=2}^{k+1} a_{k+1,r}\int_0^{\infty} \mu(g^{-1}(t))^{-1} e^{-\thtw g^{-1}(t)-\lambda_r t}dt.
\end{align*}
As an example, consider Kingman's $n$-coalescent with exponential growth with rate $\rho=0$. From \cite{griffiths1994sampling}, we see that $\mu(t)=e^{\rho t}$ and thus $g^{-1}(t)=\rho^{-1}\log(1+\rho t)$. This leads to moments 
$$E(S^k)=1-\thtw\sum_{r=2}^{k+1} a_{k+1,r} \int_0^{\infty} (1+\rho t)^{-\frac{\theta}{2\rho}-1}e^{-\binom{r}{2}t}.$$
\end{rem} 

\vspace{2mm}

Now, consider coalescents $(\Pi^{(n)}_t)_{t\geq 0}$ with dust which stay infinite. An example for this are Dirac $n$-coalescents with $\Lambda=\delta_p$, $p\in(0,1)$ \cite{Eldon2006}. For $\Lambda$-coalescents with dust, staying infinite is equivalent to $\Lambda(\{1\})=0$.   
\begin{theo}\label{prop:convtoS_dust}
Let $O_n$ be defined for $\Lambda$-$n$-coalescents with  $\mu_{-1}<\infty$ (with dust), $P(\limsup_{n\to\infty}|\Pi^{(n)}_t|=\infty \ \forall\ t>0)=1$ and with mutation rate $\thtw$. We have $$\frac{O_n}{n}\stackrel{a.s.}{\to} S$$ for $n\to\infty$ with $S>0$ a.s.. We have
$E(S)=1-\frac{\theta}{2\mu_{-1}}\frac{a}{1-a}$ with $a=\left(1-\frac{\Lambda([0,1])}{\mu_{-1}}\right)\left(\frac{\thtw}{\thtw+\mu_{-1}}\right)$. 
\end{theo}    
\begin{proof}
From \cite[Thm. 1]{Freund2017size}, we see that the asymptotic frequency of the block of 1 forms an increasing jump-hold process $f_1$ with $f_1\stackrel{a.s.}{=}\lim_{n\to\infty} \frac{\mathcal{C}_{n,1}(t)}{n}$ with values in $[0,1]$, positive jumps and with i.i.d. $Exp(\mu_{-1})$ waiting times between jumps. It fulfills $E(f_1[k])=1-(1-\frac{\Lambda([0,1])}{\mu_{-1}})^k$, where $f_1[k]$ is the value of $f_1$ at its $k$th jump. We record between which indices of jumps $K$, $K+1\in\mathbb{N}$ of $f_1$ the waiting time $T^{(1)}_n$ for the first non-private mutation falls. From the proof of \cite[Cor. 1]{Freund2017size}, we know that there exists $n_0\in\mathbb{N}$ so that $E^{(1)}_n$ equals the time of the first jump of $f_1$ for all $n\geq n_0$ almost surely. Similarly to the proof of Theorem \ref{prop:convtoS}, we just have to trace back the first mutation after this first jump whose time of appearance does not depend on $n$. This implies that $T^{(1)}_n$ falls between the same $K,K+1$ for all $n\geq n_0$. Thus we have $\lim_{n\to\infty} n^{-1}O_n=f_1[K]$ a.s., where $f_1[K]$ is the state of $f_1$ at the $K$th jump. We only need to find the distribution of $K$. For $n\geq n_0$, $T^{(1)}_n$ is the waiting time for the first jump of $f_1$ plus an independent $Exp(\thtw)$ random variable $M^{(1)}$. Using that the waiting times $(T_{1,k})_{k\in\mathbb{N}}$ between the jumps of $f_1$ are i.i.d., we have $K=1+Y$, where $Y$ is defined by $\sum_{i=1}^{Y} T_{1,i+1} \leq M < \sum_{i=1}^{Y+1} T_{1,i+1}$ and thus $Y\stackrel{d}{=}Geo\left(\frac{\thtw}{\mu_{-1}+\thtw}\right)$ on $\mathbb{N}_0$. This yields $K\stackrel{d}{=}Geo\left(\frac{\thtw}{\mu_{-1}+\thtw}\right)$ on $\mathbb{N}$. We compute
\begin{align*}
E(S)&=\sum_{k\in\mathbb{N}}E(f_1[k])P(K=k)\\&=1-\sum_{k\in\mathbb{N}}
\left(1-\frac{\Lambda([0,1])}{\mu_{-1}}\right)^k\left(\frac{\mu_{-1}}{\mu_{-1}+\thtw}\right)^{k-1}\frac{\thtw}{\mu_{-1}+\thtw}\\
&= 1- \frac{\theta}{2\mu_{-1}}\sum_{k\in\mathbb{N}} \left(\underbrace{\left(1-\frac{\Lambda([0,1])}{\mu_{-1}}\right)\left(\frac{\thtw}{\thtw+\mu_{-1}}\right)}_{=:a}\right)^k\\ &= 1-\frac{\theta}{2\mu_{-1}}\frac{a}{1-a}
\end{align*}      
\end{proof}
\section{Recursions for the moments}
To obtain recursive formulae for the moments of $O_n$ for a $\Lambda$-$n$-coalescent, we first need to introduce
$X_n$, the size of the block of 1 at the exponential clock $M$ of rate $\frac{\theta}{2}$ in the $n$-coalescent. 

%\begin{theo}
%Let $j\geq1$.
%The $j$-th moments of $X_n$ and $O_n$ satisfy the following recursions: $E(X_1^j)=1, E(O_2^j)=2^j$ and
%\begin{equation}\label{recX}
%E(X_n^j)=\frac{\thtw}{\thtw+\lambda_{n}}+\frac{1}{\thtw+\lambda_{n}}\sum_{k=2}^{n}\binom{n}{k}\lambda_{n,k}E(X_{n-k+1}^j+\sum_{i=1}^{j}\binom{j}{i-1}\frac{X_{n-k+1}^{i}}{n-k+1}).
%\end{equation}
%and
%\begin{equation}\label{recO}
%E(O_n^j)
%=\sum_{k=2}^{n}\sum_{i=0}^ja_{i,j.k}E(X_{n-k+1}^i)
%+b_{i,j.k}E(O_{n-k+1}^i),
%\end{equation}
%where 
%$$a_{i,j,k}=\frac{\lambda_{n,k}}{\lambda_n}\binom{n-1}{k-1}\binom{j}{i}(k-1)^{j-i},$$
%$$b_{0,j,k}=\frac{-1}{n-k}\frac{\lambda_{n,k}}{\lambda_n}\binom{n-1}{k},$$
%$$b_{j,j,k}=(1+\frac{j}{n-k})\frac{\lambda_{n,k}}{\lambda_n}\binom{n-1}{k}$$
%and for $1\leq i\leq j-1$
%$$b_{i,j,k}=\frac{1}{n-k}(\binom{j}{i-1}-\binom{j}{i})\frac{\lambda_{n,k}}{\lambda_n}\binom{n-1}{k}.$$
%\end{theo}

\begin{theo}
Let $j\geq1$.
The $j$th moments of $X_n$ and $O_n$ satisfy the following recursions: $E(X_1^j)=1, E(O_2^j)=2^j$ and
\begin{equation}\label{recX}
E(X_n^j)=\frac{\thtw}{\thtw+\lambda_{n}}+\frac{1}{\thtw+\lambda_{n}}\sum_{k=2}^{n}\binom{n}{k}\lambda_{n,k}E(X_{n-k+1}^j+\sum_{i=1}^{j}a_{i-1,j,k}\frac{X_{n-k+1}^{i}}{n-k+1}).
\end{equation}
and
\begin{equation}\label{recO}
E(O_n^j)
=\sum_{k=2}^{n}\binom{n-1}{k-1}\frac{\lambda_{n,k}}{\lambda_n}\sum_{i=0}^j(a_{i,j.k}E(X_{n-k+1}^i)
+\frac{{(n-k+1)}{\bf 1}_{i=j}+b_{i,j,k}}{n-k}E(O_{n-k+1}^i)),
\end{equation}
where {$a_{-1,j,k}=0$,
$a_{i,j,k}=\binom{j}{i}(k-1)^{j-i},$
and
${b_{i,j,k}=a_{i-1,j,k}-a_{i,j,k}.}$}
\end{theo}

\begin{proof}
Our proofs rely on tracking the number of blocks involved in the first jump of the $\Lambda$-$n$-coalescent, with some additional condition(s). 
Let us first prove \eqref{recX}.
Let $T_n$ be the waiting time for the first coalescence in the $n$-coalescent which is $Exp\left(\lambda_{n}\right)$-distributed. 
Let $A_{n}$  be the event that the first block merged is part of the block of 1 stopped at $M$.
\begin{align*}
&E(X_n^j)\\=&P(M^{(1)} \leq T^{(1)}_n)+P(M^{(1)} > T^{(1)}_n)\sum_{k=2}^{n}\frac{\binom{n}{k}\lambda_{n,k}}{\lambda_n}(E(( k-1+X_{n-k+1})^j1_{A_n})+E(X_{n-k+1}^j(1-1_{A_n})))\\
=&\frac{\thtw}{\thtw+\lambda_{n}}+\frac{1}{\thtw+\lambda_{n}}\sum_{k=2}^{n}\binom{n}{k}\lambda_{n,k}
\left(E((k-1+X_{n-k+1})^j\frac{X_{n-k+1}}{n-k+1})+E(X_{n-k+1}^j(1-\frac{X_{n-k+1}}{n-k+1}))\right)\\
=&\frac{\thtw}{\thtw+\lambda_{n}}+\frac{1}{\thtw+\lambda_{n}}\sum_{k=2}^{n}\binom{n}{k}\lambda_{n,k}E(X_{n-k+1}^j+\sum_{i=1}^{j}\binom{j}{i-1}(k-1)^{j-i+1}\frac{X_{n-k+1}^{i}}{n-k+1})
\end{align*}

Now let us turn to the proof of \eqref{recO}.
Let $B_n=\{K_{n,1}(1)>1\}$ be the event that 1 does participate in the first coalescence event.
Also let $C_{n}$  be the event that the first block merged is part of the observed clade.
\begin{align*}
E(O_n^j)&=E(O_n^j1_{B_n})+E(O_n^j(1-1_{B_n}))
\\
&=\sum_{k=2}^{n}\frac{\binom{n-1}{k-1}\lambda_{n,k}}{\lambda_n}E((k-1+X_{n-k+1})^j)
\\&+\sum_{k=2}^{n}\frac{\binom{n-1}{k}\lambda_{n,k}}{\lambda_n}
E(O_{n-k+1}^j(1-1_{C_n}))
\\&+\sum_{k=2}^{n}\frac{\binom{n-1}{k}\lambda_{n,k}}{\lambda_n}
E(( k-1+O_{n-k+1})^j1_{C_n})\\
&=\sum_{k=2}^{n}\frac{\binom{n-1}{k-1}\lambda_{n,k}}{\lambda_n}E((k-1+X_{n-k+1})^j)
\\&+\sum_{k=2}^{n}\frac{\binom{n-1}{k}\lambda_{n,k}}{\lambda_n}
E(O_{n-k+1}^j(1-\frac{O_{n-k+1}-1}{n-k}))
\\&+\sum_{k=2}^{n}\frac{\binom{n-1}{k}\lambda_{n,k}}{\lambda_n}
E((k-1+O_{n-k+1})^j\frac{O_{n-k+1}-1}{n-k}).\\
\end{align*}
Expanding, we obtain the result.
\end{proof}
\begin{rem}
For Kingman's $n$-coalescent, \eqref{recX} and \eqref{recO} considerably simplify.
In particular the two first moments of $X_n$ are
$$E(X_n)=\frac{\thtw}{\thtw+\binom{n}{2}}+\frac{\binom{n}{2}}{\thtw+\binom{n}{2}}\frac{n}{n-1}E(X_{n-1})$$
and
$$E(X_n^2)=\frac{\thtw}{\thtw+\binom{n}{2}}+\frac{\binom{n}{2}}{\thtw+\binom{n}{2}}\left(\frac{n+1}{n-1}E(X_{n-1}^2)+\frac{1}{n-1}E(X_{n-1})\right).$$
and the two first moments of $O_n$ are
$$E(O_n)=\frac{2}{{n}}(1+E(X_{n-1}))-\frac{1}{n}+\frac{n-1}{n}E(O_{n-1})$$
and
$$E(O_n^2)=\frac{2}{{n}}(1+2E(X_{n-1})+E(X_{n-1}^2))-\frac{1}{n}-\frac{1}{n}E(O_{n-1})+E(O_{n-1}^2)$$
\end{rem}

\thanks{
\textit{
FF was funded by DFG grant FR 3633/2-1
through Priority Program 1590: Probabilistic Structures in Evolution.} 
}
\bibliographystyle{plain}
\bibliography{ocn.bib}

\end{document}